\documentclass[11pt]{amsart}

\usepackage{amssymb}
\usepackage[all]{xy}
\usepackage[colorlinks=true, urlcolor=rltblue, citecolor=drkgreen, linkcolor=drkred] {hyperref}
\usepackage{color}
\usepackage{pdfsync}
\definecolor{rltblue}{rgb}{0,0,0.4}
\definecolor{drkred}{rgb}{0.6,0,0}
\definecolor{drkgreen}{rgb}{0,0.4,0}
\usepackage{graphicx}
\usepackage[mathscr]{eucal}

\usepackage{lmodern}
\usepackage[capitalize]{cleveref}
\usepackage{amssymb,amsmath,amsthm}
\usepackage{mathtools, MnSymbol}
\usepackage{setspace}
\usepackage{tikz-cd}
\usepackage[T1]{fontenc}
\usepackage[utf8]{inputenc}
\usepackage{textcomp} 
\usepackage{stmaryrd}
\usepackage{datetime}
\usepackage{todonotes}
\usepackage{xcolor}
\usepackage{mdframed}
\usepackage[all]{xy}

\usepackage{thmtools}
\usepackage{enumitem}
\declaretheorem{theorem}
\declaretheorem[sibling=theorem]{lemma}
\declaretheorem[sibling=theorem]{proposition}

\declaretheorem[sibling=theorem]{definition}
\declaretheorem{question}

\DeclareMathOperator{\restrict}{\upharpoonright}

\renewcommand{\phi}{\varphi}
\newcommand{\bigwwedge}{%
  \mathop{
    \mathchoice{\bigwedge\mkern-15mu\bigwedge}
               {\bigwedge\mkern-12.5mu\bigwedge}
               {\bigwedge\mkern-12.5mu\bigwedge}
               {\bigwedge\mkern-11mu\bigwedge}
    }
}

\newcommand{\bigvvee}{%
  \mathop{
    \mathchoice{\bigvee\mkern-15mu\bigvee}
               {\bigvee\mkern-12.5mu\bigvee}
               {\bigvee\mkern-12.5mu\bigvee}
               {\bigvee\mkern-11mu\bigvee}
    }
}
\newmdtheoremenv[backgroundcolor=cyan]{theorem-prove}{Theorem}[theorem]
\newmdtheoremenv[backgroundcolor=cyan]{lemma-prove}{Lemma}[theorem]
\newmdtheoremenv[backgroundcolor=cyan]{proposition-prove}{Proposition}[theorem]

\newmdtheoremenv[backgroundcolor=yellow!40]{theorem-check}{Theorem}[theorem]
\newmdtheoremenv[backgroundcolor=yellow!40]{lemma-check}{Lemma}[theorem]
\newmdtheoremenv[backgroundcolor=yellow!40]{proposition-check}{Proposition}[theorem]

\def\hbar{{\bar{h}}}

\def\T{\mathcal{T}}

\newtheorem{thm}{Theorem}

\theoremstyle{remark}

\newtheorem{example}[thm]{Example}

\def\and{\mathrel{\&}}

\title{Hybrid Maximal Filter Spaces}
\author{David Gonzalez}

\address{Department of Mathematics, University of California, Berkeley}
\email{david\_gonzalez@berkeley.edu}

\begin{document}
\maketitle
\begin{abstract}
We introduce a new way of encoding general topology in second order arithmetic that we call hybrid maximal filter (hybrid MF) spaces.
This notion is a modification of the notion of a proper MF space introduced by Montalb\'an.
We justify the shift by showing that proper MF spaces are not able to code most topological spaces, while hybrid MF spaces can code any second countable MF space.
We then answer Montalb\'an's question about metrization of well-behaved MF spaces to this shifted context.
To be specific, we show that in stark contrast to the original MF formalization used by Mummert and Simpson, the metrization theorem can be proven for hybrid MF spaces within $\text{ACA}_0$ instead of needing $\Pi_2^1-\text{CA}_0$.
\end{abstract}

\section{Introduction}

In this article we explore the concept of general topology in second order arithmetic through a new lens.
Because a topology is naturally described as a third order object, it is not clear how to formalize its study in second order arithmetic.
The usual approach involves restricting to topologies with structure that allows you to see the topology in a second order way.
The most well studied approach looks at separable, complete metric spaces (see: \cite{Bro90},\cite{MDBook} Chapter 10).
However, many theorems of topology lack meaning in this setting.
For example, metrization and the study of not completely metrizable spaces (like order topologies) are totally excluded.

To remedy the lack of expressive power of metric spaces, there have been two main explored approaches.
The first is studying countable second-countable (CSC) spaces (eg. \cite{Dor11}, \cite{Sha20}).
In this setting everything is countable, so all aspects of the topological space can be captured using only second order arithmetic.
This system has the obvious disadvantage that it only describes countable spaces.
In particular, it is not a good place to study notions such as connectivity.

The second approach is studying countably based maximal filter (MF) spaces (e.g. \cite{MS09}).
It is this second approach that we focus on in this article.
Here, a topological space is encoded by a partially ordered set.
Each basic open is given by an element in this partially ordered set, and the points of the space are given by maximal filters in the partially ordered set.
A point is said to be in a basic open if the element corresponding to that open set is an element of the point as a maximal filter.
This approach has the advantage of defining a broad class of topological spaces that includes all separable, complete metric spaces and many spaces that are not metrizable (exactly classified in \cite{MS10}).
MF spaces as formalized here do have some undesirable properties.
For example, saying a subset of a partially ordered set is a maximal filter can be as difficult as requiring $\Pi_1^1-\text{CA}_0$ (this is not explicitly noted \cite{MS09}, but we will see a short proof of it later based on those results).
In other words, it may take $\Pi_1^1-\text{CA}_0$ to declare that a set encodes a point in your space.
For this reason, it can be difficult to compute relatively simple topological properties in the MF space setting.
For example, Mummert and Simpson \cite{MS09} showed that the set of pairs of opens such that the closure of the first is contained in the second may need $\Pi_2^1-\text{CA}_0$ to compute.

The above results are striking in the high complexity they achieve.
They also bring up a potential concern.
The difficulty of recording basic topological relations may hide the combinatorial core of a theorem or make a construction seem overly difficult to describe.
In this, results proven in the setting of MF spaces leave a stone unturned: what would the complexity of the result be without these descriptive difficulties surrounding basic topological notions like point and containment.
For example, the well known result of the metrization theorem being equivalent to $\Pi_2^1-\text{CA}_0$ in MF spaces (\cite{MS09}) ought to be far simpler if we were in a setting where basic notions were easier to compute.

The coding of an object is often an innocuous part of computability theory.
This is not the case with MF spaces; there are apparent difficulties with the chosen coding.
It is for this reason that alternate codings ought to be explored, especially if they result in different outcomes for the descriptive complexity of objects like a complete metric on the space.
This is what led Montalb\'an to propose the notion of proper MF spaces, and what leads us to investigate the notion of hybrid MF spaces.

In the next sections we will give the necessary background on reverse mathematics and MF spaces.
After this we will study proper MF spaces in general to justify the definition of hybrid MF spaces.
Next we will give a proof of the metrization theorem for hybrid MF spaces in $\text{ACA}_0$.
The last section will feature some complexity calculations and suggestions for future work.

\section{Background on Reverse Math}

Reverse mathematics aims to calibrate a proof to a particular subsystem of second order arithmetic.
Over a weak base system a typical argument shows that a mathematical theorem implies a stronger system and vice versa.
We mention the following systems in this paper (listed from weakest to strongest).
$\text{RCA}_0$ contains $\Delta_1^0$ comprehension and $\Sigma_1^0$ induction.
In other words, it allows only for the definitions of computable objects.
$\text{ACA}_0$ contains arithmetical comprehension and arithmetical induction.
This allows for the definition of arithmetical objects.
$\Pi_1^1-\text{CA}_0$ is $\text{ACA}_0$ along with $\Pi_1^1$ comprehension.
$\Pi_2^1-\text{CA}_0$ is $\text{ACA}_0$ along with $\Pi_2^1$ comprehension.
There are many texts that go over the definitions of these subsystems in detail (see e.g.\cite{Simbook}, \cite{MDBook}).

While $\text{RCA}_0$ is occasionally mentioned, $\text{ACA}_0$ is base system for this paper.
This base system is unusually strong, but this is typical when working with MF spaces. 
The reason for this is that in \cite{LM06} it was shown that the existence of a maximal filter in a poset is equivalent to $\text{ACA}_0$.
As our spaces will be coded by maximal filters in a poset, it is difficult to account for the case where there are no maximal filters at all.
Furthermore, typical objects like upward closures of sets in a poset cannot be calculated without $\text{ACA}_0$ and this is needed constantly.
In short, one cannot give an honest formalization of MF spaces outside of $\text{ACA}_0$.

Our main result concerns metrization, and it is proven in our base system.
We do not produce a reversal, as such an endeavor is not common when working in a base system.

One could understand the metrization result through a more structural lens.
The result in \cite{MS09} shows that given a regular MF space there is a metric on it, but this metric necessarily uses $\Pi_2^1$ sets in its definition.
Our metrization result is proven in $\text{ACA}_0$, so shows that any regular hybrid MF space (a concept defined later) has a metric on it and this metric is arithmetically defined.
In short, the subsystems of second order arithmetic used in this context can be understood as bounds on the complexity of defining additional metric structure on a given topological space.

\section{Background on MF Spaces} 
This section aims to give the definitions regarding MF spaces established in \cite{MS09} that will also be used in this article.
There will also be definitions of a few needed terms from general topology.

\begin{definition}
Fix a partially ordered set $P$.
A \textit{pre-filter} on $P$ is a non-empty subset $R$ with the property that for every $x,y\in R$ there is a $z\in R$ with $z\geq x$ and $z\geq y$.
A \textit{filter} on $P$ is a pre-filter that is upwards closed.
A \textit{maximal filter} is a filter that is not strictly contained in any filter.
\end{definition}

Given $S\subseteq P$ we let ucl$(S)$ denote the upwards closure of $S$.
Note that if $S$ is a pre-filter then ucl$(S)$ is a filter.
We use the concept of maximal filters to defined a maximal filter space.

\begin{definition}
Let $P$ be a countable partially ordered set.
$MF(P)$ is the set of maximal filters on $P$.
Every $p\in P$ codes a basic open set
\[ N_p:= \{m\in MF(P)\vert p\in m\}.\]
A general open set is given by $U\subseteq P$ as follows
\[ N_U:= \bigcup_{p\in U} N_p.\]
These open sets are called the \textit{poset topology} on $MF(P)$
\end{definition}

Note that in $\text{ACA}_0$ one can show that the $N_U$ are closed under unions and finite intersections.
In particular, it is justified to call these sets a topology.
Note also that we can code closed sets with the code for the open set that represents its complement.
Concepts that can be defined in a general topological space carry the same definitions in this context based on the open sets $N_U$.
For example, $m\in\text{cl}(N_U)$ if every $q\in m$ has a common extension with some $p\in U$.

The metrization theorem concerns regular spaces.
\begin{definition}
An MF space is \textit{regular} if for every $x\in MF(P)$ and $p\in P$ with $x\in N_p$ there is a $q\in P$ with $n\in N_q$ and $\text{cl}(N_q)\subseteq N_p$.
\end{definition}

We will also mention the more technical notion of strong regularity.
\begin{definition}
An MF space is \textit{strongly regular} if there is a sequence of subsets $\langle D_p\vert p\in P\rangle$ with the property that $N_p=\bigcup_{q\in R_p} N_q$ and $\text{cl}(N_q)\subseteq N_p$ whenever $q\in D_p$.
\end{definition}

Given a strongly regular space, and $p\in P$ we will always use $D_p$ to denote the set in the definition above.
In a classical sense it is not difficult to see that all regular spaces are strongly regular.
In fact, one can make this observation in $\Pi_2^1-\text{CA}_0$ (see \cite{MS09} Lemma 4.1).
One only must enumerate the set $p,q$ such that $\text{cl}(N_q)\subseteq N_p$.
One of the central insights of \cite{MS09} is that this set of pairs can in fact be $\Pi_2^1$-complete.
In other words, showing strong regularity from regularity requires the strength of $\Pi_2^1-\text{CA}_0$.
This insight is what ultimately leads to showing that the following metrization theorem is equivalent to $\Pi_2^1-\text{CA}_0$.

\begin{definition}
\textit{MFMT} is the statement that every regular MF space is homeomorphic to a complete separable metric space.
\end{definition}

The notion of homeomorphism between an MF space and a complete separable metric space is similar to the one we use in Definition \ref{fcode}.
It is not central to this exposition, so we omit the exact notion used in \cite{MS09} here.

We will briefly mention the notion of a \textit{normal} topological space which is one in which every pair $N_p,N_q$ with $\text{cl}(N_q)\cap \text{cl}(N_p)=\emptyset$ has $r,t$ such that $\text{cl}(N_q)\subseteq N_r$, $\text{cl}(N_p)\subseteq N_t$ and $N_t\cap N_r = \emptyset$.

The final topology definition that will be used is that of paracompactness.
\begin{definition}
A cover of a topological space is \textit{point finite} if every point in the space has a neighborhood that only intersects finitely many of the elements of the cover.
A topological space is \textit{paracompact} if every cover admits a point finite subcover.
\end{definition}

The fact that every complete metric space is paracompact was formalized as Theorem II.7.2 in \cite{Simbook}.
We will return to this notion when it is needed at the end of the final section.

\section{Proper MF Spaces and Hybrid MF Spaces}

Montalb\'an \cite{MonQs} made the following definition.

\begin{definition}
A partial order $P$ defines a \textit{proper MF space} if for $p,q\in P$, $p\leq q$ if and only if $N_p\subseteq N_q$ and if $p$ and $q$ are comparable there is an $r\in P$ such that $N_r=N_p\cap N_q$.
\end{definition}

The hope was that these sorts of spaces would be easier to work with than general MF spaces. In particular, every MF space would be codeable as a proper MF space, but issues like definability of points would be simplified.
This is what prompted Montalb\'an in \cite{MonQs} to ask his 15th question: what is the reverse mathematical strength of the complete metrization of proper MF spaces?

Unfortunately, we observe below that the notion of a proper MF space is not a suitable formulation for further study.
Contrary to the suspicions of Montalb\'an, proper MF spaces cannot code every MF space.

In the proposition below we use $p\wedge q$ to denote the $r$ such that $N_r=N_p\cap N_q$ in the proper MF space.

\begin{proposition}\label{proper}
The topology of a proper MF space given by the elements of its poset is a clopen basis.
\end{proposition}

\begin{proof}
Let $P$ be the partial order that gives rise to the MF space. 
Consider a basic open set coded by $y\in P$.
Say there is a maximal filter $M\subseteq P$ with $M\in \partial(N_y)$.
Consider the set 
$$N:=\{p\in P\vert p\in M \lor \exists m\in M ~ p=m\wedge y\}.$$
First note that $N$ does not contain the empty set.
This is because every open set containing $M$ must intersect $N_y$.
Therefore, any $m\in M$ has a non-empty meet with $y$.

We claim that $N$ is a pre-filter.
Note that for $m,n\in M$, we have that $m\wedge y, n\wedge y\in N$ so $(m\wedge y) \wedge (n\wedge y) = (m\wedge n)\wedge y\in N$ as $m\wedge n \in M$.
Similarly,  $m, n\wedge y \in N$ so $m \wedge (n\wedge y) = (m\wedge n)\wedge y\in N$ as $m\wedge n \in M$.
Of course, $m\wedge n \in M\subseteq N$.
This means that $\text{ucl}(N)$ is a filter.
Note that $\text{ucl}(N)$ contains $M$.
By maximality of $M$, this means that $M=\text{ucl}(N)$.
In particular, $y\in M$ and therefore $M\in N_y$

This means that the boundary of $N_y$ is contained in $N_y$, and therefore it is closed.
\end{proof}

\begin{example}
$[0,1]$ is not a proper MF space.
\end{example}

This follows from Proposition \ref{proper} and the fact that $[0,1]$ is connected, so it cannot have a clopen basis.

One method that this example discounts for finding posets including meet that code a particular topological space is taking the meet-semilattice of open sets given to us by the basis of the topology.
More explicitly, if we take the standard countable basis for $[0,1]$ of rational radius balls with rational centers, this meet-semilattice of open sets does not code $[0,1]$ in the sense of an MF space.
For example, the set of opens containing $1/2$ is not a maximal filter.
In fact, it can be extended to a maximal filter in two different ways.
Either the open rational intervals ending in $1/2$ can be added or the open rational intervals starting with $1/2$ can be added.
It can be observed that this topology is quite distinct from that of $[0,1]$ as every rational number is duplicated in this way.

This excursion leads us to the conclusion that the basic open sets in an MF space actually serve two distinct roles that cannot be unified into a single object:
\begin{enumerate}
	\item Like in all spaces, they encode topological information via a lattice of open sets,
	\item They code points via maximal filters on a selected subposet of the above lattice.
\end{enumerate}
In many cases the subposet cannot be a an induced subposet as in the example above.
Either this, or the subposet lacks meets or joins.
The lattice and poset are actually quite far apart from each other.
It seems natural then that a more complete description and encoding of an MF space should actually include both of these objects and how they interact.

This leads to the following "hybrid" approach to the definition of an MF space in second order arithmetic.

\begin{definition}
A \textit{hybrid MF space} is a pseudo-complemented distributive lattice $(L,\leq_L,\wedge,\vee,0,1,^{c})$ with a distinguished subposet $(P,\leq_p)$ such that:
\begin{enumerate}
	\item $x\leq_Ly$ for $x,y\in P$ if and only if every maximal filter in $P$ containing $x$ also contains $y$.
	\item Every non-zero element $x\in L$ is $\leq_L$ above infinitely many elements of $P$. For ease of notation we call this set $B(x)\subseteq P$.
	\item $x\leq_Ly$ for $x,y\in L$ if and only if $B(x)\subseteq B(y)$.
	\item $B(x \wedge y)=B(x)\cap B(y)$.
	\item $B(x\vee y)= B(x) \cup B(y)$.
	\item If $B(x \wedge y)=\emptyset$ then $x\wedge y=0$.
	\item If every maximal filter passes through $B(x \vee y)$ then $x\vee y=1$.
	\item $B(x^c)$ is the set of all $y$ such that $x\wedge y=0$.
\end{enumerate}
\end{definition}

We say that a maximal filter $M$ is in a basic open $p\in P$ if $p\in M$.
For $x\in L$ we say that $M$ is in $x$ if there is some $p\in P$ with $x\geq_Lp$ and $p\in M$.
$M$ is contained union of basic open sets coded by $W\subseteq L$ if there is an $x\in W$ with $M$ in $x$.

It can be shown that any partial order $P$ can be upgraded to a hybrid MF space using $\Pi_2^1-\text{CA}_0$.
The most difficult portion is the definition of $\leq_L$ among the elements of $P$.
The definition given above is naturally $\Pi_2^1$ as $\leq_L = \{(x,y) \vert \forall M\subseteq P ~ MaxFilt(M)\to (x\in M \to y\in M)\}$ and $MaxFilt(M)$ is $\Pi_1^1$.

To be concrete, it is not difficult to define a hybrid MF representation of the previously problematic space $[0,1]$.
$L$ is given by the collection of unions of finitely many rational subintervals of $[0,1]$.
$P$ is given by the rational subintervals of $[0,1]$.
$\leq_P$ is given by the relation that states that the closure of a an interval sits inside the larger interval (i.e. the right and left end points are contained in the larger interval).
Lemma 4.2 of \cite{MS09} demonstrates that this representation as an ordinary MF space is as desired.
$\leq_L$ is defined by containment which is easily computable.
$\wedge,\vee,0,1$ and $^{c}$ are similarly easy to compute.

The following is essentially immediate from the work of Mummert and Simpson (e.g. it follows from Lemma \ref{strongReg} below).

\begin{lemma}[\cite{MS09}]
The statement that every partial order $P$ can be extended to a hybrid MF-space is equivalent over $\text{ACA}_0$ to $\Pi_2^1-\text{CA}_0$.
\end{lemma}

The question now becomes what sort of results can be proven in the hybrid formalism.
In the next section we explore Montalb\'an's 15th question from \cite{MonQs} in the hybrid context.
We assess the strength of the metrization theorem in a setting where basic topological definitions are easily definable, as Montalb\'an desired.

However, before we move forward we note here some basic properties of the pseudo-complement that will prove helpful later on.
Note that the definition of pseudo-complement can be naturally extended to open sets coded by reals.
We let $U^c=\{p\in L \vert \forall q\in U q\vee p = 0\}$, a set definable in $\text{ACA}_0$.
All of these properties are provable in $\text{ACA}_0$.
In fact, the proof is quite straightforward and precisely the same as Lemma 4.5 in \cite{MS09} (a more detailed treatment can be found in \cite{MThesis} Lemma 4.3.11).

\begin{lemma}\label{complement}
($\text{ACA}_0$) Fix an open set $U$ coded by a real number in a hybrid MF space.
\begin{enumerate}
	\item The entire space is a disjoint union of $\text{cl}(U)$ and $U^c$. In particular, every point is in exactly one of these sets.
	\item There is no point in both $U$ and $\text{cl}(U^c)$.
\end{enumerate}
\end{lemma}

We will use the above lemma freely in the proofs of the next section.

\section{Metrization}

In MF spaces, there is a deep difference between regular spaces and the made-for-reverse-math notion of strongly regular spaces. In fact, showing that regular spaces are strongly regular is equivalent to $\Pi_2^1-\text{CA}_0$. In hybrid MF spaces, this distinction does not exist and can be seen in the base theory of $\text{ACA}_0$.

\begin{lemma}\label{strongReg}
($\text{ACA}_0$) Every regular hybrid MF space is strongly regular.
\end{lemma}

\begin{proof}
We only need to form the set of $p,q\in P$ such that $\text{cl}(N_p)\subseteq N_q$.
This is easily done by checking the condition $p^c\vee q=1$.
\end{proof}

Another previously discussed aspect of MF spaces is that it is difficult to define the set of points.
In general, the points are $\Pi_1^1$ to define.
\begin{lemma}
The predicate $MaxFilt(X)$ which states that $X$ is a maximum filter in a poset is $\Pi_1^1$ in general.
\end{lemma}

\begin{proof}
In \cite{MS09} they show that the set of $p,q$ such that $\text{cl}(N_q)\subseteq N_p$ can be $\Pi_2^1$  complete.
We can write this set in terms of the $MaxFilt(X)$ predicate as follows:
\[ \text{cl}(N_q)\subseteq N_p \iff \forall X ~ \big(MaxFilt(X)\land \forall \ell\in X \exists r ~ r\geq\ell \land r\geq q \big) \to p\in X.\]
If the above is $\Pi_2^1$  complete then $MaxFilt(X)$ must be strictly $\Pi_1^1$.
\end{proof}

The above is true even if we restrict to case of regular spaces, the matter of interest in the metrization theorem.
A major contrast in the hybrid approach is that the set of points is definable in the base theory of $\text{ACA}_0$ due to the following observation.

\begin{lemma}\label{arithPoint}
($\text{RCA}_0$) There is an arithmetic formula that is true exactly on the maximal filters of a given regular hybrid MF space.
\end{lemma}

\begin{proof}
We claim that a subset of $P$ is a maximal filter if and only if it is a filter and for any $p\in P-M$ one of the following 2 cases hold:
\begin{enumerate}
	\item there is an $m\in M$ such that $p\wedge m= \emptyset$,
	\item for every $r\in D_p$ there is an $n\in M$ such that $r\wedge n = \emptyset$.
\end{enumerate}

We first show that if one of these conditions hold of every $p\in P-M$, then $M$ is maximal.
If $M$ was not maximal, there would be a maximal filter $N$ containing some point $p\in P-M$.
If Condition 1 holds of this point, then $N$ cannot be a filter as there is $p,m\in N$ with no common extension.
This is an immediate contradiction.
Say instead that Condition 2 holds of the point $p$.
As the $r\in D_q$ cover $N_p$ by construction, in order theoretic terms, any maximal filter containing $p$ must contain one of the $r\in D_q$.
Therefore, $N$ contains one of these $r$.
However, $N$ cannot contain such an $r$ as there is an $n\in M$ such that $r\wedge n = \emptyset$.

We now show that if $M$ is maximal then one of these conditions holds.
If $M$ is maximal, then it is a point in the space, and $p$ represents an open set that does not contain $M$.
It is immediate that Condition 1 holds exactly when $M$ is not in the closure of $N_p$.
Therefore, we need to check that if $M$ is in the boundary of $N_p$, then Condition 2 holds.
Consider an $r\in D_p$.
As $\text{cl}(N_r)\subseteq N_p$, $M\not\in \text{cl}(N_r)$.
Therefore, there is  $n\in M$ such that $r\wedge n = \emptyset$, as desired.
\end{proof}

We will use the arithmetic definition of maximal filter moving forward in the proof.

We now prove metrizability in the hybrid setting.
The outline of the proof is similar to the one in \cite{MS09}, but there are differing details throughout.
These details are critical as they move the complexity of the proof down from $\Pi_2^1-\text{CA}_0$ to $\text{ACA}_0$.
We start with the regularity implies normality lemma.

\begin{lemma}
($\text{ACA}_0$) Let $X$ be a hybrid MF space. There are functionals $\nu_1,\nu_2$ with the following property. If $U$ and $V$ are open subsets of $X$ and $\text{cl}(U)\cap \text{cl}(V)=\emptyset$, then $\nu_1(U,V)$ and $\nu_2(U,V)$ are disjoint open subsets of $X$ such that $\text{cl}(U)\subseteq\nu_1(U,V)$ and $\text{cl}(V)\subseteq\nu_2(U,V)$.
\end{lemma}

\begin{proof}
For an open set $W$, let $\{e(W,n)\}_{n\in\omega}$ be an enumeration of the union of $D_p$ for $p\in W$ in increasing order. Let 
$$U(n)=e(U^c,n)\cap\bigcap_{i\leq n} (e(V^c,i))^c$$
$$V(n)=e(V^c,n)\cap\bigcap_{i\leq n} (e(U^c,i))^c$$
be the corresponding elements of $L$.

Define
$$\nu_1(U,V)=\{V(n)\}_{n\in\omega},$$
$$\nu_2(U,V)=\{C(n)\}_{n\in\omega}.$$

The classical proof that these functionals have the desired properties now goes through in $\text{ACA}_0$.
To be specific, let $x\in \text{cl}(U)$.
As $\text{cl}(U)\cap \text{cl}(V)=\emptyset$, this means that $x\in e(V^c,n)$ for some $n$. 
Furthermore, if $x\not\in e(U^c,i)^c$ this means that every open set containing $x$ intersects $e(U^c,i)$.
However, this means that $x\in \text{cl}(U^c,i)\subseteq U^c$.
This is a contradiction to the fact that $x\in \text{cl}(U)$.
Therefore, $x\in V(n)$ for some $n$.
A similar and symmetric proof demonstrates this point for $V$.

To see that the two open sets are disjoint, note that if $x\in U(n)$ for some $n$ it is in $e(U^c,n)$ and each of the $(e(V^c,i))^c$ for $i\leq n$.
This means that $x$ is not in $V(m)$ for $m\leq n$ as $(e(U^c,m))^c\cap e(U^c,m)= \emptyset$.
Furthermore, $x$ is not in $V(m)$ for $m\geq n$ as $(e(V^c,m))^c\cap e(V^c,m)= \emptyset$.
This completes the proof.

\end{proof}

\begin{lemma}
($\text{ACA}_0$) Fix $p\in P$ and $q\in R_p$. Let D be the set of dyadic rational numbers in the interval $[0, 1]$. There is a map $g_{p,q}$ from $D$ to coded open subsets of $X$ such that $g_{p,q}(0) = \{q\}$, $g_{p,q}(1) = \{p\}$, and $\text{cl}(g_{p,q}(k)) \subseteq g_{p,q}(k')$ whenever $k < k'$.
\end{lemma}

\begin{proof}
Let $U_0= \{q\}$ and $U_1= \{p\}$. We define the rest of the $U_k$ for dyadic $k$ by induction.
Throughout the induction we will maintain that for dyadic rationals $k<k'$, $\text{cl}(U_k)\subseteq U_{k'}$.
Let us note that this is an arithmetic condition.
In particular, we can check if $U_k^c\cup U_{k'}=1$.
This means that we can preform our induction within $\text{ACA}_0$.

We assume that the map has already been defined for all dyadic $k$ with reduced form $\frac{a}{2^m}$.
We show how to define the map on dyadic $\ell$ with reduced form $\frac{b}{2^{m+1}}$.
Select $k<k'$ with reduced form $\frac{a}{2^m}$ that are the closest to $\ell$.
We have that $\ell = \frac{k+k'}{2}$.
By induction, we assume that $\text{cl}(U_k)\subseteq U_{k'}$.
In other words, we have that $\text{cl}(U_k)\cap \text{cl}(U_{k'}^c)=\emptyset$, as otherwise there would be a point in both $U_{k'}$ and $U^c_{k'}$.
This means that we define the $U_\ell$ as $U_\ell= \nu_1(U_k,U_{k'}^c)$.
By the previous lemma this can be done in $\text{ACA}_0$.
Furthermore, we are guaranteed that $\text{cl}(U_k)\subseteq U_\ell$.
Lastly, we claim that $\text{cl}(U_\ell)$ is contained in $U_{k'}$.
Note that $\nu_2(U_k,U_{k'}^c)$ is an open disjoint from $U_{\ell}$ by construction.
Furthermore, $\nu_2(U_k,U_{k'}^c)$ contains the closure of $U_{k'}^c$.
Therefore, every point not in  $\nu_2(U_k,U_{k'}^c)$ must be in $U_{k'}$, and in particular every point in $U_{\ell}$ is in $U_{k'}$ as required.
\end{proof}

The next three results have arguments that are quite similar to those found in \cite{MS09} (and indeed they are also quite similar to the classical arguments).
There are not many details different from the argument for ordinary MF spaces and hybrid MF spaces because the arguments produced in \cite{MS09} already go through at the lower complexity level of $\text{ACA}_0$.
The proofs of these intermediate lemmas are still produced here, but not in full detail; one ought to examine \cite{MS09} Lemmas 4.7-4.10 or \cite{MThesis} Lemmas 4.3.12, 4.3.13, and 4.3.18 for further details.

\begin{lemma}
($\text{ACA}_0$) There is a sequence 
$\langle f_{p,q}\vert p\in P,q\in R_p\rangle$
of continuous functions from $P \to [0,\infty)$ such that for all $p \in P$ and $q \in R_p$ we have $f_{p,q} (N_q) = 0$ and $f_{p,q} (X - N_p) = 1$.
\end{lemma}

\begin{proof}
We can form the set $\langle g_{p,q}\vert p\in P,q\in R_p\rangle$ for all $p$ and $q$ by the previous lemma.
We extend $ g_{p,q}$ to all positive dyadic rationals by letting $ g_{p,q}(k)$ be the whole space if $k> 1$.
As in the previous lemma, $U_k$ is meant to indicate the open set coded by $ g_{p,q}(k)$.
Let 
\[f_{p,q}=\inf\{k\in D^+\vert x\in U_k\}.\]
It is clear that these $f_{p,q}$ satisfy the desired conclusion.
This can be executed in $\text{ACA}_0$.
$f_{p,q}(x)<k$ if and only if there is some $k'<k$ with $x\in U_{k'}$ if and only if there is $r\in P$ with $r\in x\cap U_{k'}$.
To see $f_{p,q}(x)>k$ you similarly check for a $k''>k$ with $x\in U_{k'}^c$ which can be seen by looking for an  $r\in P$ with $r\in x\cap U_{k''}^c$.
This is enough to define the output of $f$.
\end{proof}

From this, we get the desired metrizability.

\begin{proposition}\label{metriz}
($\text{ACA}_0$) Every hybrid MF space is metrizable.
\end{proposition}

\begin{proof}
Reindex the function from the above lemma to be $f_n$ and let
$$d(x,y)=\sum_{n\in\omega} 2^{-n} \vert f_n(x) - f_n(y)\vert.$$
The classical proof of Uryshon's metrization theorem is straightforward and shows that this is a metric.
\end{proof}

At this point we pause to note that the above development answers Question 4.1 from \cite{MS09}, in a certain sense. 
They ask what can be said about the reverse mathematical strength of stating that every strongly regular MF space has a (not necessarily complete) metric.
By moving to a different representation meant to easily capture topological notions like strong regularity, we are able to show this in the base theory of $\text{ACA}_0$.
Or course, in a stricter sense their question remains open.

We now wish to show that hybrid MF spaces are completely metrizable.
It is in fact this theorem that was shown in \cite{MS09} to be equivalent to $\Pi_2^1-\text{CA}_0$.

We use the following lemma to demonstrate this, which states that in $\text{RCA}_0$ every $G_\delta$ subset of a completely metrizable space is completely metrizable.

\begin{lemma}\label{gdelta}
($\text{RCA}_0$) If $(U_i)$ is a sequence of open sets in a complete, separable metric space $(\hat{A},D)$, the there is a complete metric $d'$ such on $\bigcap_i U_i$ that agrees with the subspace topology on $\bigcap_i U_i$. Furthermore, if $D\subseteq \bigcap_i U_i$ is computable and dense then this topology is canonically isomorphic to $\hat{D}$.
\end{lemma}

\begin{proof}
This follows from the classical proof see e.g. \cite{KBook} Theorem 3.11.
More details for the context of reverse mathematics can be found in \cite{MThesis} Lemma 4.3.18.
\end{proof}

Given a hybrid MF space $X$ we let $A\subseteq X$ be a computable dense set definable in $\text{ACA}_0$.
Note that this can be found; given non-empty $p\in P$ we find a maximal filter containing $p$ by recursion.
Let $p_0=p$, and let $p_{i+1}$ be the least element of $P$ that has non-zero intersection with $p_0,\cdots, p_i$ and non-zero intersection with some sequence of elements in $D_{p_0}\times\cdots\times D_{p_i}$.
By Lemma \ref{arithPoint} this is enough to guarantee that we build a maximal filter in the limit.
We show first that $X$ is isometrically embeddable in $\hat{A}$, the completion of $A$, and then we show that the image of this embedding is $G_\delta$.

To formalize the idea of a continuous map between an MF space and a metric space we take the following definition.

\begin{definition}\label{fcode}
A continuous function code $F$ between a hybrid MF space based on the poset $P$ and a complete metric space given by a dense set $A$ is given by a set of pairs $\langle p,\langle a,r\rangle \rangle \in P\times A\times \mathbb{R}$. $F$ induces a partial function from $MF(P)$ to $\hat{A}$. We say that $F$ is defined at $m\in MF(P)$ if for every sequence of elements in $m_i \in m$  there is $\langle m_i,\langle a_i,r_i\rangle \rangle \in F$ with $r_i\to 0$. Furthermore, for every such $\langle m_i,\langle a_i,r_i\rangle \rangle \in F$ with $r_i\to 0$, the $a_i$ are a Cauchy sequence in the same equivalence class. In this case we say that $F$ sends $m$ to this unique equivalence class of Cauchy sequences. A coded continuous function is one that is a total function induced by a continuous function code.
\end{definition}

It is not too difficult to see in ZFC that every continuous function between a hybrid MF space and a complete metric space is induced by such a code.
It is enough to check that given such a function $f$, $\{\langle p,\langle a,r\rangle \rangle \in P\times A\times \mathbb{R} \vert f(N_p)\subseteq B_r(a)\}$ is a continuous function code for $f$.

\begin{lemma}\label{embed}
($\text{ACA}_0$) If $X$ is a metrizable countably based hybrid MF space, there is a separable complete metric space $\hat{A}$ and a continuous, open bijection $h$ between $X$ and a dense subset of $A$. Furthermore, $h$ is an isometry.
\end{lemma}

\begin{proof}
Define the metric $d$ on $X$ with Proposition \ref{metriz}.
Let $A$ be a countable dense subset of $X$.
For each $p\in P$, 
\[ diam(p):= \sup \{d(a,b) \vert a,b\in A\cap N_p\}.\]
Let $H$ be the set of all $\langle p,\langle a,r\rangle \rangle $ such that $diam(p)<r$ and $a\in N_p$.

Consider $m$ a maximal filter in $P$.
We first find a $\langle m_i,\langle a_i,r_i\rangle \rangle \in H$ with $r_i\to 0$.
Let $b_i$ be a sequence of points in $A$ that converge to $m$.
Because $X$ is metrizable, there is some $p\in m$ such that $diam(p)\leq 2^{-n-1}$.
Let $b_{i(n)}$ be such that $b_{i(n)}\in N_p$.
This yields that $\langle m_i,\langle b_i,2^{-n}\rangle \rangle \in H$.
As $n$ was arbitrary this gives the needed sequence.

Now consider two distinct  $\langle m_i,\langle a_i,r_i\rangle \rangle ,\langle m_i',\langle a_i',r_i'\rangle \rangle \in H$.
Say that $r_i\to 0$ and $r_i'\to 0$ and they are both decreasing without loss of generality.
By definition, $d(a_i, m)<{r_i}$ and $d(a_i', m)<{r_i'}$.
This means that for all $j\geq i$ $d(a_i, a_j)<2{r_i}$ and $d(a_i', a_j')<2{r_i'}$.
Or, what is the same, the sequences are Cauchy.
Furthermore, $d(a_i,a_i')< {r_i+r_i'}$ and therefore tends to $0$ too.
This means that the sequences $a_i$ and $a_i'$ must be in the same equivalence class as desired.

It is clear that the $h$ that $H$ codes is an isometry because it is an isometry on the dense set $A$ shared between the spaces.
As $h$ is an isometry, it is open on its dense image.
\end{proof}

We will also use the Paracompactness Lemma from Theorem II.7.2 in \cite{Simbook}.

\begin{lemma}
($\text{RCA}_0$) Given a complete, separable metric space $(A,d)$, there is a computable functional $\Phi$ that, given a sequence of opens $U_i$, returns a point finite refinement $V_i$ 
\end{lemma}

\begin{theorem}
($\text{ACA}_0$) Every hybrid MF space is completely metrizable.
\end{theorem}

\begin{proof}
We wish to construct a sequence of open subsets of $\hat{A}$ that witness that $\text{im}(h)$ from Lemma \ref{embed} is $G_\delta$.
From here we will appeal to Lemma \ref{gdelta} to complete the proof.

By our previous results, we obtain a metric $d$ on $X$ that induces the topology, along with a set of witnesses $R_p$ for $p\in P$ for the strong regularity of $X$.
We construct countable sequences of open sets and will form a tree of these sequences.
The odd entries in the sequences, denoted $W_i$, will be themselves sequences of balls in the metric topology centered on $a\in A$.
The even entries will be elements $q_i\in P$ that represent basic open sets in the MF topology.
In this notation, the sequences will look like $(W_0,q_0,\cdots, W_{n}, q_n)$ or $(W_0,q_0,\cdots, q_n, W_{n+1})$.
To define the precise tree we are working with, we put the following restrictions on the sequences:
\begin{enumerate}
	\item For every $i$ and $q_i$, there is a basic open ball $B$ enumerated in $W_i$ such that $N_q\subseteq B$.
	\item For every $i$, $q_{i+1}\leq_Pq_i$.
	\item For every $i$, $\text{diam}(q_i)<2^{-i}$.
	\item $W_0$ is an element of the point finite refinement of the canonical ${B(a,1)}_{a\in A}$ covering of $X$.
	\item $W_{i+1}$ is an element of the point finite refinement of the cover of the element $W_i$ given by the set of $N_{q_i}$ such that the $q_i$ satisfy Conditions 1-3.
\end{enumerate}

It is clear that Conditions 1-4 are arithmetic.
Condition 5 is also arithmetic, but requires more explanation.
In the statement of Condition 5, it is implicitly assumed that the set of $N_{q_i}$ satisfying  Conditions 1-3 cover $W_i$.
We start by showing this.
Consider a point $x\in W_i$.
Take a point $a\in A$ with $d(x,a)<2^{-i}$.
The set  $B(a,2^{-i})\cap W_i$ is open, and is therefore a union of $N_r$ for $r\in P$.
In particular, for some $c$ with $N_{c}\subseteq B(a,2^{-i})\cap W_i$, $x\in N_{c}$.
If $i>0$, then the $W_i$ came from a point finite refinement of $q_{i-1}$, so, $W_i\subseteq N_{q_{i-1}}$.
This means that, considering $x$ as a filter, $q_{i-1}, c\in x$.
Take $q_i$ to be a common extension of $q_{i-1}$ and $c$.
It can now be see that $x\in N_{q_i}$ with all of the desired properties.

Strictly speaking, we can only apply the paracompactness functional to balls in the space, not the opens given by the maximal filter space.
That being said, we can obtain a description for $N_q$ in terms of balls arithmetically, which fixes this issue.
In particular, given $N_q$, let $$U_q=\{B\vert \exists p\in R_q ~~ B\cap A\subseteq N_p\cap A\}.$$
Mummert shows that $N_q = \bigcup_{B\in U_q} B$ in Lemma 4.3 of \cite{MS09}.
In terms of reverse mathematics, when we say "the point finite refinement of the cover of the element $W_i$ given by the set of $N_{q_i}$ such that..." we precisely mean "the point finite refinement of the cover of the element $W_i$ given by the set of $U_{q_i}$ such that..."

With this, it should be clear that the above conditions are well defined and, in fact, arithmetic.
Because the satisfaction of these properties is determined locally by looking at the preceding elements in the finite sequence, it is clear that the sequences satisfying the conditions form a countable tree definable in $\text{ACA}_0$.
We call this tree $T$.
We now define the open sets $C_i$ as the union of the $2i+1$st level of $T$.
We claim that $\text{im}(h)=\bigcap_i C_i$.

We first show that $\text{im}(h)\subseteq\bigcap_i C_i$.
Given $x\in\text{im}(h)$, we show that $\forall i  ~ x\in C_i$ by induction.
$C_0$ is a cover by definition so it contains $x$.
By the argument above, if $x\in C_i$ because $x\in W_i$, there is an extension $q_i$ with $x\in N_{q_i}$.
Furthermore, because the paracompactness functional maintains covers, $x$ must be in some $W_{i+1}$.
Therefore, $x\in C_{i+1}$ as desired.

We now show that $\bigcap_i C_i\subseteq \text{im}(h)$.
Fix $z\in \bigcap_i C_i$.
We form the tree, $T_z$, which contains all sequences $(W_0,\cdots, W_n)$ that are odd projections from elements of $T$ with $x\in W_n$.
$T_z$ is infinite as there are arbitrarily long paths, because $z\in C_i$ for all $i$.
However, $T_z$ is finitely branching.
By construction, once you fix $W_i$ the possible extensions form a point finite cover, so only finitely many of them contain $z$.
This means that only finitely many of these extensions are in $T_z$, as we claimed. 
By Konig's Lemma there is an infinite path through the tree $(W_i)_{i\in\omega}$.
By definition, $z\in \bigcap_i W_i$.
This corresponds to a sequence $q_i$ in the even projection from the tree $T$ with $z\in \bigcap_i q_i$.
Note that the $q_i$ are arbitrarily small, so $\bigcap_i q_i$ has only one point.
Now consider $m=\text{ucl}_L((q_i)_{i\in \omega})\subseteq P$.
It is not difficult to see that $m$ is a filter.
We now show that it is maximal.
If it was not maximal, there would be an $r\in P - m$ with $c\in D_r$ such that for all $i$ $c\wedge q_i\neq\emptyset$.
Let $\varepsilon=d(\bar{N_c},\bar{N_r^c})$, positive by regularity.
Take $i$ with $2^{-i}<\varepsilon$ and note that $q_i\leq_L r$.
Therefore, $r\in m$, a contradiction.
This gives that the image of $m$ under $h$ must be $z$.
Therefore, $z\in\text{im}(h)$, as desired.

Lemma \ref{gdelta} completes the proof.
\end{proof}

\section{Some Complexity Calculations and Future Work}

We define a map which demonstrates a few interesting complexity calculations.
These calculations suggest directions for future work.

\begin{definition}
Given a tree $\T\subseteq \mathbb{N}^{<\mathbb{N}}$ we define $\Phi(\T)\subseteq (\mathbb{N}\cup \{*\})^{<\mathbb{N}}$ where $*$ is a symbol not among $\mathbb{N}$.
We take $\Phi(\T) = \{\sigma^{\frown}* \vert \sigma\in\T\} \cup \T$.
\end{definition}

Note that given a tree, we can understand it as a partial ordering via the inclusion relation.
In particular, every tree also codes an MF space.
Furthermore, in $\text{ACA}_0$ this MF space can be extended to a hybrid MF space.
It is easy to confirm $\leq_P$ and $\leq_L$ coincide as the maximal filters are the maximal paths through the tree and the points in the tree just contain each path that includes them.
One can then extend the structure to include $0,1,\wedge,\vee,^c$ in a direct way.
In what follows, we consider trees as a hybrid MF space in the manner described above.

\begin{lemma}\label{Treg}
Every $\T\subseteq \mathbb{N}^{<\mathbb{N}}$ represents a regular topology.
\end{lemma}

\begin{proof}
It was observed above that $\leq_P$ and $\leq_L$ coincide in the topology of $\T$.
Furthermore, it can be confirmed that given $\sigma,\tau\in \T$ their meet is the longest common initial segment of $\sigma$ and $\tau$.
The meet is of $\sigma$ and $\tau$ are then inside of the poset $P$.
This means that $\T$ represents a proper MF space.
By Proposition \ref{proper}, this means that the basis is clopen.

In other words for $x\in N_p$, it is enough to note that $\text{Cl}(N_p)\subseteq N_p$ to see that there is a witness to regularity for the pair $x,p$ (namely $p$ itself).
\end{proof}

\begin{proposition}
The set of discrete, regular hybrid MF spaces is $\mathbf{\Pi}_1^1$ complete.
\end{proposition}

\begin{proof}
First we note that there is a $\mathbf{\Pi}_1^1$ definition.
Being discrete for a regular space is equivalent to saying there are no limit points; in other words, no sequence of points converges to any particular point.
This can be expressed as
\[ \forall X, \{Y_i\}_{i\in\omega} \big(\bigwwedge MaxFilt(Y_i) \land MaxFilt(X)\big) \to \exists p\in P ~ p\in X \land \bigwwedge p\not\in Y_i.\]
By Lemma \ref{arithPoint} , $MaxFilt(X)$ can be expressed as an arithmetic formula, so the above expression is $\Pi_1^1$ in $P$ as desired.

We now note that $\Phi(\T)$ is discrete if and only if $\T$ is well-founded.
Note that if $\T$ is well-founded, so is $\Phi(\T)$.
In any well-founded tree the maximal filters are finite, and are therefore isolated by the unique maximal string in the filter.
Therefore, any well-founded tree gives rise to a discrete space.
If $\T$ is not well-founded let $A$ be a path in $T$.
It is straightforward to confirm that $A$ is a limit point of $(A\restrict n)^\frown *$, so $\Phi(\T)$ is not discrete.
\end{proof}

The above calculation is made exact by the fact that $MaxFilt(X)$ can be expressed as an arithmetic formula.
It does not yield a similar result with ordinary MF spaces as the given definition of discreteness would be $\Pi_2^1$ in $P$.
This is a good example of how moving to the hybrid formalism can allow for greater expressive power in a way that facilitates calculations above the level of ACA$_0$.
The same comments apply to the example below.

\begin{proposition}
The set of covers of a regular hybrid MF spaces are $\mathbf{\Pi}_1^1$ complete.
\end{proposition}

\begin{proof}
First we note that there is a $\mathbf{\Pi}_1^1$ definition.
To say that a set of opens $\{p_i\}_{i\in\omega}$ covers a space is to say every maximal filter contains one of these points.
This can be expressed as
\[\forall X ~ MaxFilt(X) \to \bigvvee p_i\in X. \]
By Lemma \ref{arithPoint} , $MaxFilt(X)$ can be expressed as an arithmetic formula, so the above expression is $\Pi_1^1$ in $\{p_i\}_{i\in\omega}$ as desired.

We now note that $\{\sigma^{\frown}* \vert \sigma\in\T\}$ covers $\Phi(\T)$ if and only if $\T$ is well-founded.
If $\T$ is well-founded the maximal filters in $\Phi(\T)$ are all maximal finite segments.
By construction, these all end in $*$; i.e. there is a point in the filter among $\{\sigma^{\frown}* \vert \sigma\in\T\}$.
Therefore, $\{\sigma^{\frown}* \vert \sigma\in\T\}$ covers $\Phi(\T)$.
If $\T$ is If $\T$ is not well-founded let $A$ be a path in $T$.
It is straightforward to confirm that $A$ is a point in $\Phi(\T)$ that is not within $\{\sigma^{\frown}* \vert \sigma\in\T\}$.
\end{proof}

This proposition begins to aim at the theory of compactness in which covers play a central role.
The following remains an open and motivating question for further development of this area.

\begin{question}
What is the reverse mathematical strength of the one point compactification theorem?
\end{question}

\end{document}